\pdfoutput=1

\documentclass{amsart}
\usepackage{url,epsfig}
\usepackage{amsmath,amsthm,amsfonts}

\usepackage[norelsize]{algorithm2e}  
\newtheorem{theorem}{Theorem}
\newtheorem{observation}{Observation}
\newtheorem{definition}{Definition}
\def \mathR {\mathbb{R}}
\def \Vor {\mbox{Vor}}
\def \Pow {\mbox{Pow}}
\def \Del {\mbox{Del}}
\def \Reg {\mbox{Reg}}

\begin{document}
%%-----------------------------
%%      the top matter
%%-----------------------------
\title{A Numerical Algorithm for \\ $L_2$ Semi-Discrete Optimal Transport in 3D}% At most 5 thanks
\author{Bruno L\'evy}\thanks{Inria Nancy Grand-Est and LORIA, rue du Jardin Botanique, 54500 Vandoeuvre, France}
\date{04/01/2014}
\begin{abstract}
  This paper introduces a numerical algorithm to compute the $L_2$ optimal transport
map between two measures $\mu$ and $\nu$, where $\mu$ derives from a
density $\rho$ defined as a piecewise linear function (supported by a
tetrahedral mesh), and where $\nu$ is a sum of Dirac masses. 

  I first give an elementary presentation of some known results on optimal transport and then observe a relation
with another problem (optimal sampling). This relation gives simple arguments to study
the objective functions that characterize both problems. 

  I then propose a practical algorithm to compute the optimal transport map between a
piecewise linear density and a sum of Dirac masses in 3D. In this semi-discrete setting,
Aurenhammer et.al [\emph{8th Symposium on Computational Geometry conf. proc.}, ACM (1992)] 
showed that the optimal transport map is determined by the weights of a power diagram. 
The optimal weights are computed by minimizing a convex objective function with a quasi-Newton method. 
To evaluate the value and gradient of this objective function, I propose an efficient and robust 
algorithm, that computes at each iteration the intersection between a power diagram and the tetrahedral mesh 
that defines the measure $\mu$. 

  The numerical algorithm is experimented and evaluated on several datasets, with up to hundred thousands
tetrahedra and one million Dirac masses.
\end{abstract}
%
%

%% 49M15=Newton-type methods
%% 35J96=Elliptic Monge-Ampere eqn
%% 68U05=Computer Graphics/ computational geometry [Computer Science]
%% 65D18=Computer Graphics/ computational geometry [Numerical analysis]
\maketitle

\section*{Introduction}

%% Definition of optimal transport, importance, recent applications in graphics

Optimal Transportation, initially studied by Monge \cite{Monge1784}, is 
a very general problem formulation that can be used as a model for a wide range of
applications domains. In particular, it is a natural formulation
for several fundamental questions in Computer Graphics
\cite{DBLP:journals/focm/Memoli11,DBLP:journals/cgf/Merigot11,DBLP:journals/tog/BonneelPPH11}

  This article proposes a practical algorithm to compute the optimal transport
map between two measures $\mu$ and $\nu$, where $\mu$ derives from a
density $\rho$ defined as a piecewise linear function (supported by a
tetrahedral mesh), and where $\nu$ is a sum of Dirac masses. Possible applications
comprise measuring the (approximated) Wasserstein distance between two
shapes and deforming a 3D shape onto another one (3D morphing).  \\

  I first review some known results about optimal transport in Section \ref{sect:OT}, 
  its relation with power diagrams \cite{DBLP:conf/compgeom/AurenhammerHA92,DBLP:journals/cgf/Merigot11}
  in Section \ref{sect:semidiscreteOT}
  and observe some connections with another problem (optimal sampling
  \cite{Lloyd82leastsquares,Du:1999:CVT:340312.340319}).  The
  structure of the objective function minimized by both problems is
  very similar, this allows reusing known results for both
  functions. This gives a simple argument to easily compute the
  gradient of the quantization noise power minimized by optimal
  sampling, and this gives the second order continuity of the
  objective function minimized in semi-discrete optimal transport (see Section \ref{sect:CVTandOT}). \\

  I then propose a practical algorithm to compute the optimal
  transport map between a piecewise linear density and a sum of Dirac masses
  in 3D (Section \ref{sect:numerics}). This means determining the weights of a power diagram,
  obtained as the unique minimizer of a convex function
  \cite{DBLP:conf/compgeom/AurenhammerHA92}.  Following the approach
  in \cite{DBLP:journals/cgf/Merigot11}, to optimize this function, I
  use a quasi-Newton solver combined with a multilevel
  algorithm. Adapting the approach to the 3D setting requires an
  efficient method to compute the intersection between a power diagram
  and the tetrahedral mesh that defines the density $\mu$. \\

  To compute these intersections, the algorithm presented here
  simultaneously traverses the tetrahedral mesh and the power
  diagram (Section \ref{sect:RVD}). The required geometric predicates are implemented in both
  standard floating point precision and arbitrary precision, using
  arithmetic filtering \cite{meyer:inria-00344297}, expansion
  arithmetics \cite{DBLP:conf/compgeom/Shewchuk96} and symbolic
  perturbation \cite{Edelsbrunner90simulationof}. Both predicates and
  power diagram construction algorithm are available in PCK (Predicate
  Construction Kit) part of my publically available ``geogram''
  programming
  library\footnote{\url{http://gforge.inria.fr/projects/geogram/}}. \\

  The algorithm was experimented and evaluated on several datasets 
  %, with up to hundred thousands tetrahedra and Dirac masses. 
  (Section \ref{sect:results}).

\section{Optimal Transport: an Elementary Introduction}
\label{sect:OT}

This section, inspired by \cite{OTON}, \cite{OTintro}, \cite{MAEintro} and \cite{OTuserguide},
presents an introduction to optimal transport. It stays at an elementary level that corresponds
to what I have understood and that keeps computer implementation in mind. 

\subsection{The initial formulation by Monge}

  The problem of Optimal Transport was first introduced and studied
by Monge \cite{Monge1784}. With modern notations, it can be stated as follows~:

\begin{equation*}
  (M) \ 
  \begin{array}{l}
    \mbox{given $\Omega$ a Borel set and two measures $\mu$ and $\nu$ on $\Omega$ such that 
      $\mu(\Omega) = \nu(\Omega)$,} \\
  \mbox{find } T: \Omega \rightarrow \Omega \mbox{ such that}
     \left\{
     \begin{array}{cl}
       (C1) & \nu = T\sharp\mu \\
       (C2) & \int_\Omega c(x,T(x)) d\mu \mbox{ is minimal} 
     \end{array}
     \right.
   \end{array}
\end{equation*}
where $c$ denotes a convex distance function.
In the first constraint $(C1)$, $T\sharp\mu$ denotes the pushforward
of $\mu$ by $T$, defined by $T\sharp\mu(X) = \mu(T^{-1}(X))$ for any
Borel (i.e. measurable) subset $X$ of $\Omega$. In other words, the constraint $(C1)$ 
means that $T$ should preserve the mass of any measurable subset of $\Omega$. 
The functional in $(C2)$ has a non-symmetric structure, that
makes it difficult to study the existence for problem $(M)$. 

\begin{figure}
\centerline{\includegraphics[width=100mm]{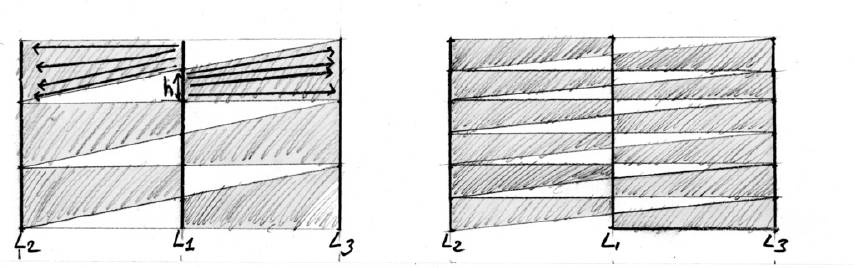}}
\caption{A classical illustration of the existence problem with Monge's 
 formulation: there is no optimal transport map
 from a segment $L_1$ to two parallel segments $L_2$ and $L_3$ (it is always possible
 to find a better one by replacing $h$ with $h/2$)} 
\label{fig:monge_pb}
\end{figure}

The non-symmetry comes from the constraint that $T$ should be a map. It makes
it possible to merge mass but not to split mass. This difficulty is
illustrated in Figure \ref{fig:monge_pb}. Suppose you want to find the
optimal transport from one vertical segment $L_1$ to two parallel
segments $L_2$ and $L_3$. It is possible to split $L_1$ into segments
of length $h$ mapped to $L_2$ and $L_3$ in alternance (Figure \ref{fig:monge_pb} left). For any length $h$, it
is always possible to find a better map, i.e. with a lower value of
the functional in $(C2)$, by splitting $L_1$ into smaller segments 
(Figure \ref{fig:monge_pb} right), therefore problem (M) does not have
a solution within the set of admissible maps. This problem occurs whenever
the source measure $\mu$ has mass concentrated on sets with
zero geometric measure (like $L_1$). 

\subsection{The relaxation of Kantorovich for Monge's problem}
To overcome this difficulty, Kantorovich proposed a relaxation of problem (M)
where mass can be both splitted and merged. The idea consists of manipulating
measures on $\Omega \times \Omega$ as follows~:

\begin{equation*}
(K) \quad \quad
   \begin{array}{l}
     \mbox{min} \left\{ \int\limits_{\Omega \times \Omega} c(x,y) d \gamma \ | \ \gamma \in \Pi(\mu,\nu) \right\} \\[5mm]
   \mbox{where } \Pi(\mu,\nu) = \{ \gamma \in {\mathcal{P}}(\Omega \times \Omega) \ | \
     (P_1)\sharp\gamma = \mu \ ; \ (P_2)\sharp\gamma = \nu \}
   \end{array}
\end{equation*}
where $(P_1)$ and $(P_2)$ denote the two projections $(x,y) \in \Omega \times \Omega \mapsto x$ and 
$(x,y) \in \Omega \times \Omega \mapsto y$ respectively.

The pushforwards of the two projections $(P_1)\sharp\gamma$ and $(P_2)\sharp\gamma$ are called the
marginals of $\gamma$. The probability measures $\gamma$ in $\Pi(\mu,\nu)$, i.e. that have $\mu$
and $\nu$ as marginals, are called \emph{transport plans}. Among the transport plans, those that
are in the form $(Id \times T)\sharp \mu$ correspond to a transport map $T$~:

\begin{observation}
  If $(Id \times T)\sharp\mu \in \pi(\mu,\nu)$, then $T$ pushes $\mu$ to $\nu$.
\end{observation}
\begin{proof}
 $(Id \times T)\sharp\mu$ belongs to $\pi(\mu,\nu)$, therefore 
  $(P_2)\sharp(Id \times T)\sharp \mu = \nu$, \\ or
  $\left((P_2)\circ(Id \times T)\right)\sharp \mu = \nu$, thus
  $T\sharp\mu = \nu$
\end{proof}

With this observation, for transport plans of the form $\gamma = (Id \times T)\sharp \mu$, (K) becomes
$$
     \mbox{min} \left\{ \int\limits_{\Omega \times \Omega} c(x,y)d\left( (Id \times T)\sharp \mu \right) \right\} 
\quad = \quad
     \mbox{min} \left\{ \int\limits_\Omega c(x,T(x)) d \mu \right) 
$$

To help intuition, four examples of transport plans in 1D are depicted in Figure
\ref{fig:OTplans}. The measure $\gamma$ on $\Omega \times \Omega$ is
non-zero on subsets that contain points $(x,y)$ such that mass is
transported from $x$ to $y$. The transport plans in the first two
examples are in the form $(Id \times T)\sharp \mu$, i.e. they are
derived from a transport map\footnote{For the second one (B), the
  transport map is not defined in the center of the segment, but it is not
  a problem since there is no mass concentrated there.}. The third and
fourth ones do not admit a transport map, because they split a
Dirac mass. The optimal transport plan for the case shown in Figure
\ref{fig:monge_pb} is of the same nature. It is not in the form $(Id
\times T)\sharp \mu$ because it splits the mass concentrated in $L_1$
into $L_2$ and $L_3$.

\begin{figure}
\includegraphics[width=130mm]{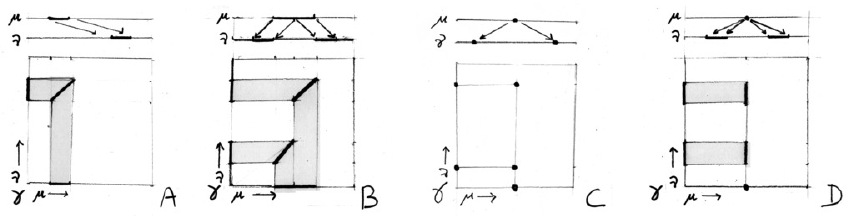}
\caption{Four examples of transport plans in 1D. A: a segment
  is translated. B: a segment is splitted into two segments. C: a
  Dirac mass is splitted into two Dirac masses. D: a Dirac mass is splitted into two segments. 
  The first two ones (A and B) are in the form $(Id \times T)\sharp \mu$ where $T$ is a transport map,
  whereas the third and fourth ones (C and D) are not, because they both split a Dirac mass.} 
\label{fig:OTplans}
\end{figure}

At this point, a standard approach to tackle the existence problem is
to find some regularity in both the functional and space of admissible
transport plans, i.e. proving that the functional is smooth enough and
finding a compact set of admissible transport plans. Since the set of
admissible transport plans contains at least the product measure $\mu
\otimes \nu$, it is non-empty, and existence can be proved using a
topological argument that exploits the smoothness of the functional
and the compactness of the set. Once the existence of a transport plan
is proved, an interesting question is whether there exists a transport
map that corresponds to this transport plan. Unfortunately, problem (K) does not
directly exhibit the properties required by this path of reasoning. However, 
one can observe that (K) is a linearly constrained optimization 
problem. This calls for studying the dual formulation, as done by Kantorovich. This
dual formulation has a nice structure, that allows answering the
questions above (existence of a transport plan, and whether there is a
transport map that corresponds to this transport plan when it exists).

\subsection{The dual formulation of Kantorovich}

The dual formulation can be stated as follows\footnote{
  Showing the equivalence with problem (K) requires some care, the
 reader is referred to \cite{OTON} chapter 5. Note that \cite{OTON} uses
 a slightly different definition (with $\phi - \psi$ instead of $\phi + \psi$), that
 makes the detailed argument simpler but that breaks symmetry between $\phi$ and $\psi$.
 Since I stay at an elementary level, I prefer to keep the symmetry.
}~:

\begin{equation*}
   (D): \quad \quad \mbox{max} \left\{
           \int\limits_\Omega \phi d\mu + \int\limits_\Omega \psi d\nu \ | \ 
            \begin{array}{ll}
           (C1) & \phi \in L^1(\mu); \psi \in L^1(\nu); \\
           (C2) & \phi(x) + \psi(y) \le c(x,y) \forall (x,y) \in \Omega \times \Omega
            \end{array}
     \right\}
\end{equation*}

Following the classical image that gives some intuition about this formula, imagine now that you 
are hiring a transport company to do the job for you. The company has a special way of calculating
the price: the function $\phi(x)$ corresponds to
what they charge you for loading at $x$, and $\psi(y)$ what they charge for unloading at $y$. 
The company tries to maximize profit (therefore is looking for a max instead of a min), but they 
cannot charge you more than what it will cost you if you do the job yourself $(C2)$. \\

The existence for $(D)$ is difficult to study, since the class of admissible functions that satisfy
$(C1)$ and $(C2)$ is non-compact. However, more structure in the problem can be revealed by referring
to the notion of \emph{c-transform}, that exhibits a class of admissible functions with regularity~:

\begin{definition}
  Given a function ${\mathcal{X}}: \Omega \rightarrow \bar{\mathbb{R}}$, the c-transform ${\mathcal{X}}^c$ is 
  defined by~:
$$
   {\mathcal{X}}^c := \inf\limits_{x \in \Omega} c(x,y) - {\mathcal{X}}(x)
$$
\begin{itemize}
   \item If for a function $\phi$ there exists a function ${\mathcal{X}}$ such that $\phi = {\mathcal{X}}^c$, 
     then $\phi$ is said to be \emph{c-concave};
   \item ${\bf \Psi}_c(\Omega)$ denotes the set of c-concave functions on $\Omega$.
\end{itemize}
\end{definition}

It is now possible to make two observations, that allow us to restrict ourselves to
the class of c-concave functions for the possible choices for $\phi$ and $\psi$~:

\begin{observation}
  If $(\phi,\psi)$ is admissible for $(D)$, then $(\phi,\phi^c)$ is also admissible.
\end{observation}
\begin{proof}
$$
\begin{array}{l}
    \left\{
        \begin{array}{l}
           \forall(x,y) \in \Omega \times \Omega, \phi(x) + \psi(y) \le c(x,y) \\
           \phi^c(y) = \inf\limits_{x \in \Omega} c(x,y) - \phi(x)
        \end{array}
    \right. \\
    \begin{array}{lcl}
        \phi(x) + \phi^c(y) & = & \phi(x) + \inf_{x^\prime \in \Omega}\left(  c(x^\prime,y) - \phi(x^\prime)  \right) \\
                            & \le & \phi(x) + c(x,y) - \phi(x) \\
                            & \le & c(x,y)
    \end{array}
\end{array}
$$
\end{proof}

\begin{observation}
   If $(\phi,\psi)$ is admissible for $(D)$, then a \emph{better candidate} can be found by replacing $\psi$ with $\phi^c$~:
\end{observation}
\begin{proof}
$$
\left\{
   \begin{array}{lcl}
       \phi^c(y) & = & \inf\limits_{y \in \Omega} c(x,y) - \phi(x) \\
       \forall x \in \Omega, \psi(y) & \le & c(x,y) - \phi(x)
   \end{array}
\right. 
   \quad \quad 
   \Rightarrow
   \quad 
   \psi(y) \le \phi^c(y)
$$
\end{proof}

Therefore, we have $\min(K) \quad = \max\limits_{\psi \in {\bf \Psi}_c(\Omega)} \quad \int\limits_\Omega \psi d\mu + \int\limits_\Omega \psi^cd\nu $ \\
I will not detail here the proof for the existence, the reader is referred to \cite{OTON}, Chapter 4.
The idea is that we are now in a much better situation, 
since the class of admissible functions ${\bf \Psi}_c(\Omega)$ is compact 
(provided that we fix the value of $\Psi$ at one point of $\Omega$ to remove the 
translational invariance degree of freedom of the problem). \\

Since we have computer implementation in mind, our goal is to find a numerical algorithm to compute an optimal transport map $T$. At first sight, 
though the values of the functionals match at a solution of $(K)$ and $(D)$, it seems to be difficult to deduce $T$ from a solution to the dual problem $(D)$.
However, there is a nice relation between the dual problem $(D)$ and the initial Monge's problem $(M)$, detailed in \cite{OTON}, chapters 9 and 10. 
The main result characterizes the pairs of points $(x,y)$ that are connected by the transport plan~:
\begin{theorem}
$$
\forall (x,y) \in \partial_c \psi, \nabla \psi(x) - \nabla_x c(x,y) = 0
$$
where $\partial_c \psi = \{ (x,y) | \phi(x) + \psi(y) = c(x,y) \}$ denotes the so-called \emph{c-subdifferential} of $\psi$.
\label{thm:MongeSol}
\end{theorem}
\begin{proof}
See \cite{OTON} chapter 10. \\
I summarize the heuristic argument given at the beginning of the same chapter, that gives some intuition~:\\
Consider a point $(x,y)$ on the c-subdifferential $\partial_c \psi$, that satisfies $\phi(y) + \psi(x) = c(x,y)\ (1)$. \\
By definition, $\phi(y) = \psi^c(y) = \inf\limits_x c(x,y) - \psi(x)$, thus $\forall \tilde{x}, \phi(y) \le c(\tilde{x},y) - \psi(\tilde{x})$,
or $\phi(y) + \psi(\tilde{x}) \le c(\tilde{x},y)\ (2)$. \\
By substituting (1) into (2), one gets $\psi(\tilde{x}) - \psi(x) \le c(\tilde{x},y) - c(x,y)$ for all $\tilde{x}$.\\
Imagine now that $\tilde{x}$ follows a trajectory parameterized by $\epsilon$ and starting at $x$. One can compute
the gradient along an arbitrary direction $w$ by taking the
limit when $\epsilon$ tends to zero in the relation $\frac{\psi(\tilde{x}) - \psi(x)}{\epsilon} \le \frac{c(\tilde{x},y) - c(x,y)}{\epsilon}$.
Thus we have $\nabla \psi(x) \cdot w \le \nabla_x c(x,y) \cdot w$. The same derivation can be done with $-w$ instead of $w$, and one gets:
$\forall w, \nabla \psi(x) \cdot w = \nabla_x c(x,y) \cdot w$, thus
$\forall (x,y) \in \partial_c \psi, \nabla \psi(x) - \nabla_x c(x,y) = 0$. \\
\emph{Note: the derivations above are only formal ones and do not make a proof. The proof requires a much more careful analysis, using generalized definitions of differentiability and tools from convex analysis.}
\end{proof}

In the $L_2$ case, i.e. $c(x,y) = 1/2 \| x - y \|^2$, we have $\forall (x,y) \in \partial_c \psi, \nabla \psi(x) + y - x = 0$, 
thus, whenever the optimal transport map $T$ exists, we have
$T(x) = x - \nabla \psi(x) = \nabla (\|x\|^2/2 - \psi(x))$. Not only this gives an expression of $T$, but also it allows characterizing $T$ as the gradient of a 
\emph{convex} function, which is an interesting property since it implies that two ``transported particles'' $x_1 \mapsto T(x_1)$ and $x_2 \mapsto T(x_2)$ cannot collide, as shown below~:

%%% Convexite de \bar{Phi}

\begin{observation}
If $c(x,y)$ = $1/2 \| x - y \|^2$ and $\psi \in {\bf \Psi}_c(\Omega)$, then $\bar{\psi}: x \mapsto \bar{\psi}(x) = \| x \|^2/2 - \psi(x)$ is convex 
(it is an equivalence if $\Omega = \mathbb{R}^d$).
\label{obs:PhiConvexity}
\end{observation}
\begin{proof}
$$
\begin{array}{lcl}
   \psi(x) & = & \inf\limits_y \frac{|x -y|^2}{2} - \phi(y) \\
           & = & \inf\limits_y \frac{\|x\|^2}{2} - x \cdot y + \frac{\|y\|^2}{2} - \phi(y) \\
-\bar{\psi}(x) & = & \phi(x) - \frac{\|x\|^2}{2} = \inf\limits_y -x \cdot y + \left( \frac{\|y\|^2}{2} - \phi(y)  \right) \\
\bar{\psi}(x) & = & \sup\limits_y x \cdot y - \left( \frac{\|y\|^2}{2} - \phi(y) \right)
\end{array}
$$
The function $ x \mapsto x \cdot y - \left( \frac{\|y\|^2}{2} - \phi(y) \right)$ is linear in $x$, therefore the graph of 
$\bar{\psi}$ is the upper envelope of a family of hyperplanes, thus $\bar{\psi}$ is convex.
\end{proof}

%%% Pas de collisions des particules

\begin{observation}
Consider the trajectories of two particles parameterized by $t \in [0,1]$, $t \mapsto (1-t)x_1 + t T(x_1)$ and
$t \mapsto (1-t)x_2 + t T(x_2)$. If $x_1 \neq x_2$ and for $0 < t < 1$ the particles cannot collide.
\end{observation}
\begin{proof}
  By contradiction, suppose that you have $t \in (0,1)$ and $x_1 \neq x_2$ such that: 
$$
  \begin{array}{lcl}
   (1-t)x_1 + tT(x_1) & = & (1-t)x_2 + tT(x_2) \\
   (1-t)x_1 + t \nabla \bar{\psi}(x_1) & = & (1-t)x_2 + t \nabla \bar{\psi}(x_2) \\
   (1-t)(x_1 - x_2)  + t (\nabla \bar{\psi}(x_1) - \nabla \bar{\psi}(x_2)) & = & 0 \\
   \forall v, (1-t) v \cdot(x_1 - x_2)  + t  v \cdot (\nabla \bar{\psi}(x_1) - \nabla \bar{\psi}(x_2)) & = & 0 \\
   \mbox{   take } v = (x_1 - x_2) \\
   (1-t)\| x_1 - x_2 \|^2 + t (x_1 - x_2)\cdot(\nabla \bar{\psi}(x_1) - \nabla \bar{\psi}(x_2)) & = & 0
  \end{array}
$$
which is a contradiction since this quantity is the sum of two strictly positive numbers (
recalling the definition of the convexity of $\bar{\psi}$: 
   $\forall x_1 \neq x_2, (x_1-x_2) \cdot (\nabla \bar{\psi}(x_1) - \nabla \bar{\psi}(x_2)) > 0$
).

\end{proof}

At this point, we know that when the optimal transport map exists, it can be deduced from the function $\psi$
using the relation $T(x) = \nabla \bar{\psi} = x - \nabla \psi$. We now consider some ways of finding the 
function $\psi$.

The classical change of variable formula gives:
\begin{equation*}
   \forall B, \int_B \mu(x) d\mu = \mu(B) = \nu(T(B)) = \int_B \frac{1}{ \det JT(x)} T(x) d\nu
   \label{eqn:changevar}
\end{equation*}
where $JT$ denotes the Jacobian matrix of $T$.

If $\mu$ and $\nu$ both have a density $u$ and $v$ (i.e. $\forall B, \mu(B) = \int_B u(x)dx$ and
$\nu(B) = \int_B v(x)dx$), then one can (formally) consider (\ref{eqn:changevar}) in a pointwise manner~:
\begin{equation*}
    \forall x \in \Omega, u(x) = \frac{1}{\det JT(x)} v(T(x))\ ;
   \label{eqn:pointwise}
\end{equation*}

injecting $T=\nabla\bar{\psi}$ and $JT = H \bar{\psi}$ in (\ref{eqn:pointwise}) gives:
\begin{equation}
\forall x \in \Omega, u(x) = \frac{1}{\det H \bar{\psi}(x)} v (\nabla \bar{\psi}(x))
\label{eqn:MAE}
\end{equation}
where $H\bar{\psi}$ denotes the Hessian of $\bar{\psi}$. Equation
\ref{eqn:MAE} is known as the \emph{Monge-Amp\`ere} equation. It is a
highly non-linear equation, and its solution when it exists often has
singularities. It is similar to the eikonal equation that
characterizes the distance function and that has a singularity on the
medial axis.  Note that the derivations above are only formal,
studying the solutions of the Monge-Amp\`ere equation requires using
more elaborate tools, and several types of weak solutions can be
defined (viscosity solutions, solutions in the sense of Brenier,
\ldots). \\

Still keeping computer implementation in mind, one may consider three
different problem settings~:

\begin{itemize}
\item {\bf continuous:} if $\mu$ and $\nu$ have a density $u$ and $v$, it is possible
to numerically solve the Monge-Amp\`ere equation, as done in \cite{ACFM:BB:2000} and \cite{papadakis:hal-00816211};
\item {\bf discrete:} if both $\mu$ and $\nu$ are discrete (sums of Dirac masses), then
finding the optimal transport plan becomes an assignment problem, that can be solved with
some variants of linear programming techniques (see the survey in \cite{AP:BDM:2009});
\item {\bf semi-discrete:} if $\mu$ has a density and $\nu$ is discrete (sum of Dirac masses), then
  an optimal transport map exists. It has interesting connections with notions of computational 
  geometry. The remainder of this paper considers this problem setting. 
\end{itemize}

\subsection{The semi-discrete case}
\label{sect:semidiscreteOT}

I now consider that $\mu$ has a density $u$, and that
$\nu = \sum_{i=1}^k \nu_i \delta_{p_i}$ is a sum of $k$ Dirac masses,
that satisfies $\nu(\Omega) = \sum_{i=1}^k \nu_i = \mu(\Omega)$. Whenever $T$ exists, 
the pre-images of the Dirac masses $T^{-1}(p_i)$ partition $\Omega$ almost everywhere\footnote{
  except on a subset of measure 0 on the common boundaries of the parts.}. 
This subsection reviews the main results in \cite{DBLP:conf/compgeom/AurenhammerHA92},
showing that this partition corresponds to a geometrical structure called a power diagram.
Interestingly, from the point of view of computer implementation, the proof directly
leads to a numerical algorithm, as experimented in 2D in \cite{DBLP:journals/cgf/Merigot11}
and in 3D further in this paper.

\begin{definition}
Given a set $P$ of $k$ points $p_i$ in $\mathR^d$ and a set $W$ of $k$ real numbers $w_i$,
the \emph{Voronoi diagram} $\Vor(P)$ and the \emph{power diagram} $\Pow_W(P)$ are defined
as follows~:
\begin{itemize}
\item The Voronoi diagram $\Vor(P)$ is the
partition of $\mathR^d$ into the subsets $\Vor(p_i)$ defined by~: \\
   $\Vor(p_i) := \{ x | \|x-p_i\|^2 < \|x - p_j\|^2 \quad \forall j \neq i\}$;
\item the power diagram $\Pow_W(P)$ is the
partition of $\mathR^d$ into the subsets $\Pow_W(p_i)$ defined by~: \\
$\Pow_W(p_i) := \{ x | \|x-p_i\|^2 - w_i < \|x - p_j\|^2 - w_j \quad \forall j \neq i\}$;
\item the map $T_W$ defined by $\forall i, \forall p \in \Pow_W(p_i), T_W(p) = p_i$ is called
the \emph{assignment defined by the power diagram} $\Pow_W(P)$.
\end{itemize}
\end{definition}

It can be shown that the assignment defined by a power diagram is an optimal transport map
(the main argument of the proof is sketched further). Then one needs to determine - when it is 
possible\footnote{We will see further that it is always possible in this setting.} - 
the parameters of this power diagram (i.e. the weights) that realize the optimal transport towards a \emph{given}
discrete target measure $\nu$.  Intuitively, a power diagram may be thought-of as a 
generalization of the Voronoi diagram,
with additional ``tuning buttons'' represented by the weights $w_i$. Changing the weight 
$w_i$ associated with a point $p_i$ influences the area and the measure $\mu(\Pow_W(p_i))$ of its power cell 
(the higher the weight, the larger the power cell). Though the relation between the weights and the measures
of the power cells is non-trivial\footnote{Misleadingly, the term 'weight' 
seems similar to 'mass', but both notions are not directly related.}, it is well behaved, and as shown below, 
one can prove the existence and uniqueness of a set of weights such that the measure of each power cell $\mu(\Pow_W(p_i))$
matches a prescribed value $\nu_i$. In this case, the prescribed measures $\nu_i$ are referred to as 
\emph{capacity constraints}, and the power diagram is said to be \emph{adapted} to the capacity constraints. 
At this point, since we already know that the assignment defined by a power 
diagram is an optimal transport map, then we are done (i.e. the assignment defined by the power diagram is the optimal
transport map that we are looking for). I shall now give more details about the proofs of the
two parts of the reasoning.

%, i.e.:
%\begin{enumerate}
%\item given a set of points $P$ and a set of weights $W$, the assignment $T_{P,W}$ yielded by the 
% power diagram is an optimal transport map;
%\item given a measure $\mu$ with density, a set of points $(p_i)$ and the target masses $\nu_i$
%  such that $\sum \nu_i = \mu(\Omega)$, there exists a unique set of weights $W$ such that 
%   $\mu(\Pow_W(p_i)) = \nu_i$.
%\end{enumerate}

\begin{figure}[t]
    \begin{minipage}{0.49\linewidth}
      \centerline{
%         \hspace{5mm}
         \includegraphics[width=53mm]{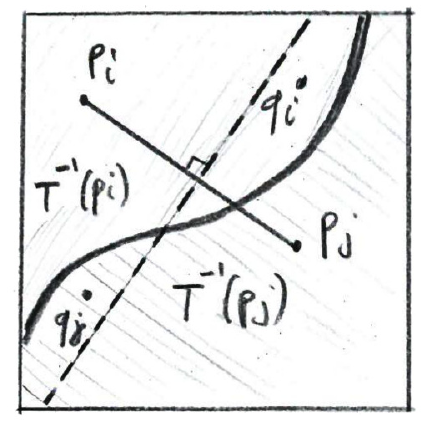}
%         \hspace{5mm}
      }
      \caption{
        Illustration of the (by contradiction) argument that the common boundary between the pre-images
        of $p_i$ and $p_j$ is contained by a straight line orthogonal to $[p_i, p_j]$.
      }
      \label{fig:StraightBoundary}
    \end{minipage}
    \begin{minipage}{0.49\linewidth}
      \centerline{
        \includegraphics[width=53mm]{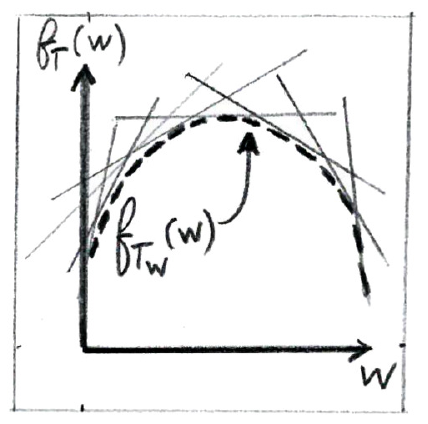}
      }
      \caption{The weight vector that defines an optimal transport map can be found as the maximizer of a convex function,
         defined as the lower envelope of a family of linear functions.}
      \label{fig:LowerEnveloppe}
    \end{minipage}
\end{figure}

\begin{theorem}
Given a set of points $P$ and a set of weights $W$, the assignment $T_{P,W}$ defined by the 
 power diagram is an optimal transport map.
\end{theorem}
\begin{proof}

  I give here the main idea of the proof (see \cite{DBLP:conf/compgeom/AurenhammerHA92} for the 
complete one).
The main argument is that if $T$ is an optimal transport map, then the common 
boundary of the pre-images $T^{-1}(p_i)$
and $T^{-1}(p_j)$ of two Dirac masses is a straight line orthogonal to the segment $[p_i, p_j]$.
The argument, obtained by contradiction, is illustrated in Figure \ref{fig:StraightBoundary}. Suppose
that the common boundary between the pre-images $T^{-1}(p_i)$ and $T^{-1}(p_j)$ is not a straight line (thick
curve in the figure), then one can find a straight line orthogonal to the segment $[p_i, p_j]$ that has
an intersection with the common boundary (dashed line in the figure), and two points $q_i$ and $q_j$ located as 
shown in the figure. Then, it is clear (by the Pythagorean theorem) that re-assigning $q_j$ to $T^{-1}(p_i)$ and $q_i$ to $T^{-1}(p_j)$ lowers the
transport cost, which contradicts the initial assumption. It is then possible to establish that
the pre-images correspond to power cells, by invoking some properties of power diagrams \cite{DBLP:journals/siamcomp/Aurenhammer87}.
\end{proof}

\begin{theorem}
  Given a measure $\mu$ with density, a set of points $(p_i)$ and prescribed masses $\nu_i$
  such that $\sum \nu_i = \mu(\Omega)$, there exists a weights vector $W$ such that 
  $\mu(\Pow_W(p_i)) = \nu_i$.
  \label{thm:SemiDiscrete}
\end{theorem}
\begin{proof}
Consider the function $f_T(W) = \int_\Omega \| x - T(x) \|^2 - w_{T(x)} d\mu$, where $T: \Omega \rightarrow P$ is 
an \emph{arbitrary} assignment. One can observe that:
\begin{itemize}
\item If the assignment $T$ is fixed, $f_T(W) = \int_\Omega \| x - T(x) \|^2 d\mu - \sum_{i=1}^{k} w_i \mu(T^{-1}(p_i))$ is
  affine in $W$. In Figure \ref{fig:LowerEnveloppe}, the graph of $f_T(W)$ for a fixed assignment $T$ corresponds to
  one of the straight lines (note that in the figure, the ``W axis'' symbolizes $k$ coordinates);
\item we now consider a fixed value of $W$ and different assignments $T$. Among all the possible $T$'s, it is clear
that $f_T(W)$ is minimized by $T_W$, the assignment defined by the power diagram with weights $W$ (the definition
of the power cell minimizes at each point of $\Omega$ the integrand in the equation of $f_T(W)$).
\end{itemize}

Now take $T = T_W$ in $f_T(W)$, in other words, consider the function
$f_{T_W}(W)$. Its graph, depicted as a dashed curve in Figure
\ref{fig:LowerEnveloppe}, is the lower envelope of a family of
hyperplanes, thus it is a concave function, with a single maximum. For the
next steps of the proof, we now need to compute the gradient $\nabla_W f_{T_W}(W)$. Note that
when computing the variations of $f_{T_W}(W)$, both the
argument $W$ of $f$ and the parameter $T_W$ change, making the
computations quite involved. When $T_W$ changes, the power cells change, and one
needs to compute integrals over varying domains. However, it is possible to drastically
simplify computations by using the \emph{envelope theorem}. Given a parameterized family of 
functions $f_T(W)$ (in our case, the parameter is $T$), 
whenever the gradient of $\nabla_W f_{T_W}(W)$ exists, it is equal to the gradient $\nabla_W f_{T^*}(W)$ computed
at the minimizer $T^*$ ($f_{T_W}$ in our case). In other words, when computing the gradients, one can directly
use the expression of $f_T(W)$ and ignore the variations of $T$ in function of $W$. In Figure \ref{fig:LowerEnveloppe},
it means that the tangent to $f_{T_W}$ at $W$ corresponds to the (linear) graph of $f_T(W)$ with a fixed $T = T_W$.
Note that in our case, the so-called \emph{choice set}, i.e. where $T$ is chosen, is the set of all the assignments between $\Omega$ and
$P$. This requires a special version of the envelope theorem that works for such a general choice sets 
\cite{RePEc:ecm:emetrp:v:70:y:2002:i:2:p:583-601}. 

One can see that the components of the gradient correspond to the (negated) measures of the power cells~:
$$
\begin{array}{lcl}
 \frac{\partial f_{T_W}(W)}{\partial w_i} & = & \nabla_W \left( 
           \underbrace{\int_\Omega \| x - T(x) \|^2 d\mu}_{\mbox{constant}(W)} - \sum\limits_{i=1}^{k} w_i \mu(T^{-1}_W(p_i)) \right) \\[12mm]
                   & = & - \mu(T_W^{-1}(p_i)) = - \mu( \Pow_W(p_i) )
\end{array}
$$

We are now in a very good situation to establish the existence and uniqueness of the weight vector $W$ that realizes the optimal transport map. 
The idea is to use $f_{T_W}$ to construct a function $g$ that has a global maximum realized at a weight vector such that
the measures of the power cells match the prescribed measures. Consider the function $g$ 
defined by $g(W) = f_{T_W}(W) + \sum_i \nu_i w_i$. The components of the gradient of $g$ are
given by $\partial g / \partial w_i = -\mu(\Pow_W(p_i)) + \nu_i$. This function is also concave (it is 
the sum of a concave function plus a linear one), therefore it has a unique global maximum where the gradient is zero. 
Therefore, at the maximum of $g$, for each power cell, the measure $\mu(\Pow_W(p_i))$ matches the prescribed measure $\nu_i$.
\end{proof}

Besides showing the existence of a semi-discrete transport map and characterizing it as the assignment defined by a power diagram, 
the proof in Theorem \ref{thm:SemiDiscrete} directly leads to a numerical algorithm, as shown in \cite{DBLP:journals/cgf/Merigot11},
described in Section \ref{sect:numerics} further.
A similar algorithm can be obtained by starting from a discrete version of the 
Monge-Ampere equation and the characterization of $T$ as the
gradient of a piecewise linear convex function\cite{DiscreteMA}.
%More details on the numerical algorithm are given in Section \ref{sect:numerics} further in this paper.

\subsection{Relation with Kantorovich's dual formulation}
\label{sect:SemiDiscKanto}
It is interesting to see the relation between the proof of Aurenhammer et.al that does not use the formalism of optimal transport, 
and the dual formulation of optimal transport. Interestingly, one can remark that the same argument 
(lower envelope of hyperplanes) is used to establish the 
concavity of $f_{T_W}$ in Theorem \ref{thm:SemiDiscrete} and the convexity of $\bar{\psi}$ in Observation \ref{obs:PhiConvexity}. 
The relation between both formulations can be further explained if we link the Kantorovich potential $\phi$ and the
weights $w_i$ with the relation $\phi(y_i) = 1/2 w_i$. For instance, 
injecting $\phi(y_i) = 1/2 w_i$ and $c(x,y) = 1/2 \| x-y \|^2$ into
$\psi(x) = \phi^c(x) = \inf_y c(x,y) - \phi(y)$ gives $\psi(x) = 1/2 \inf_i \| x-y_i \|^2 - w_i$. This corresponds to the
definition of the power cells (intuitively, the $\inf$ in the definition of $\phi^c$ is the same as the $\inf$ in the definition
of the power cell). Now consider $T(x) = x - \nabla \psi(x)$. Still using the expression of $\psi(x)$ above, we get
$T(x) = x - 1/2 \nabla_x(\| x - y_i \|^2 - w_i) = y_i$. This connects the characterization of $T$ as the solution of 
$\nabla \phi(x) - \nabla_x c(x,y) = 0$ (Theorem \ref{thm:MongeSol}) with the characterization of $T$ as the assignment 
defined by the power diagram (Theorem \ref{thm:SemiDiscrete}). This corresponds to the point of view
developed in \cite{DiscreteMA}.
%, that establishes some connections with Minkowski problem (determining a convex 
%polytope from its facet normals).

\begin{figure}[t]
    \begin{minipage}{0.56\linewidth}
      \centerline{
         \includegraphics[width=40mm]{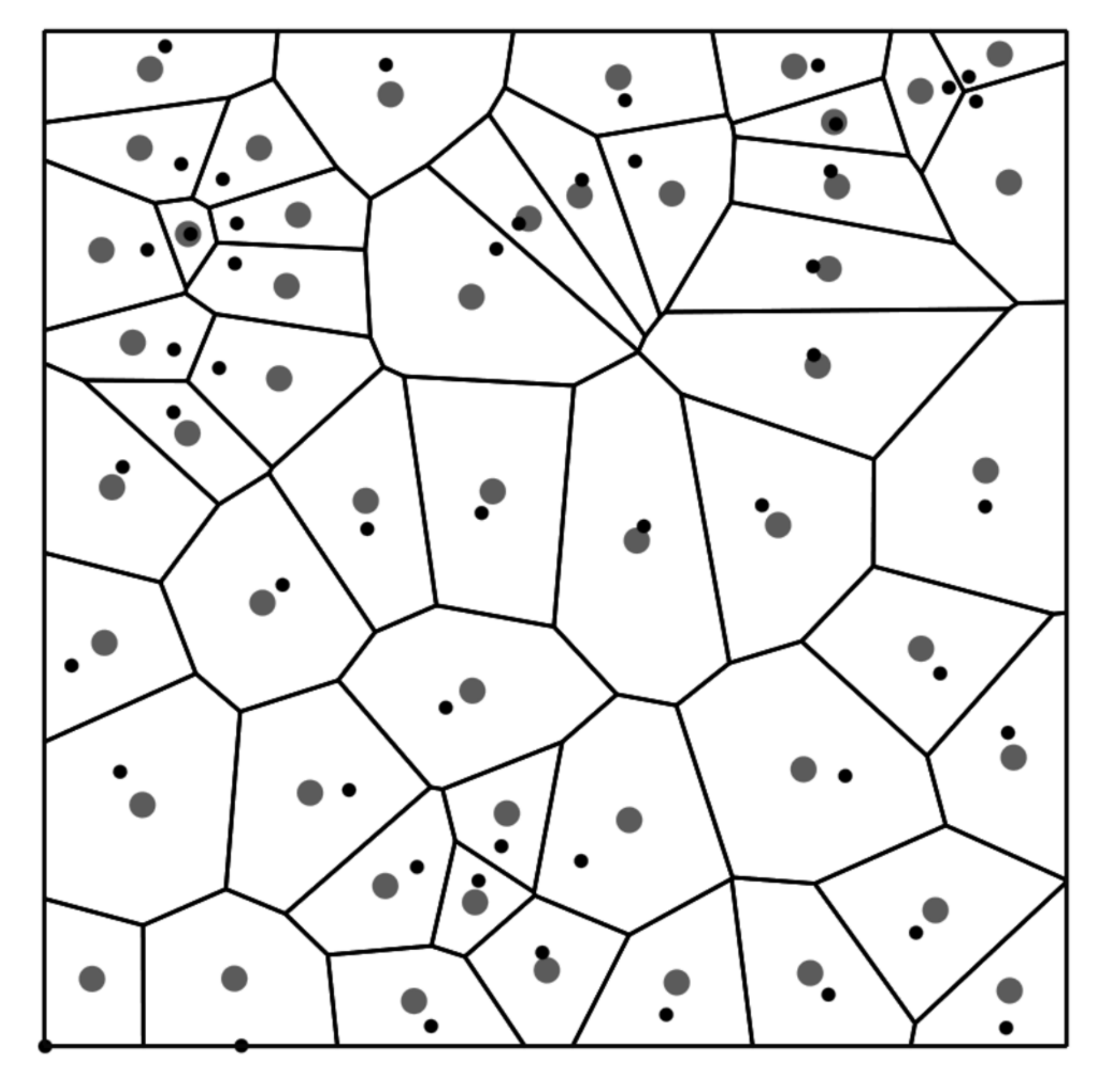}
         \includegraphics[width=40mm]{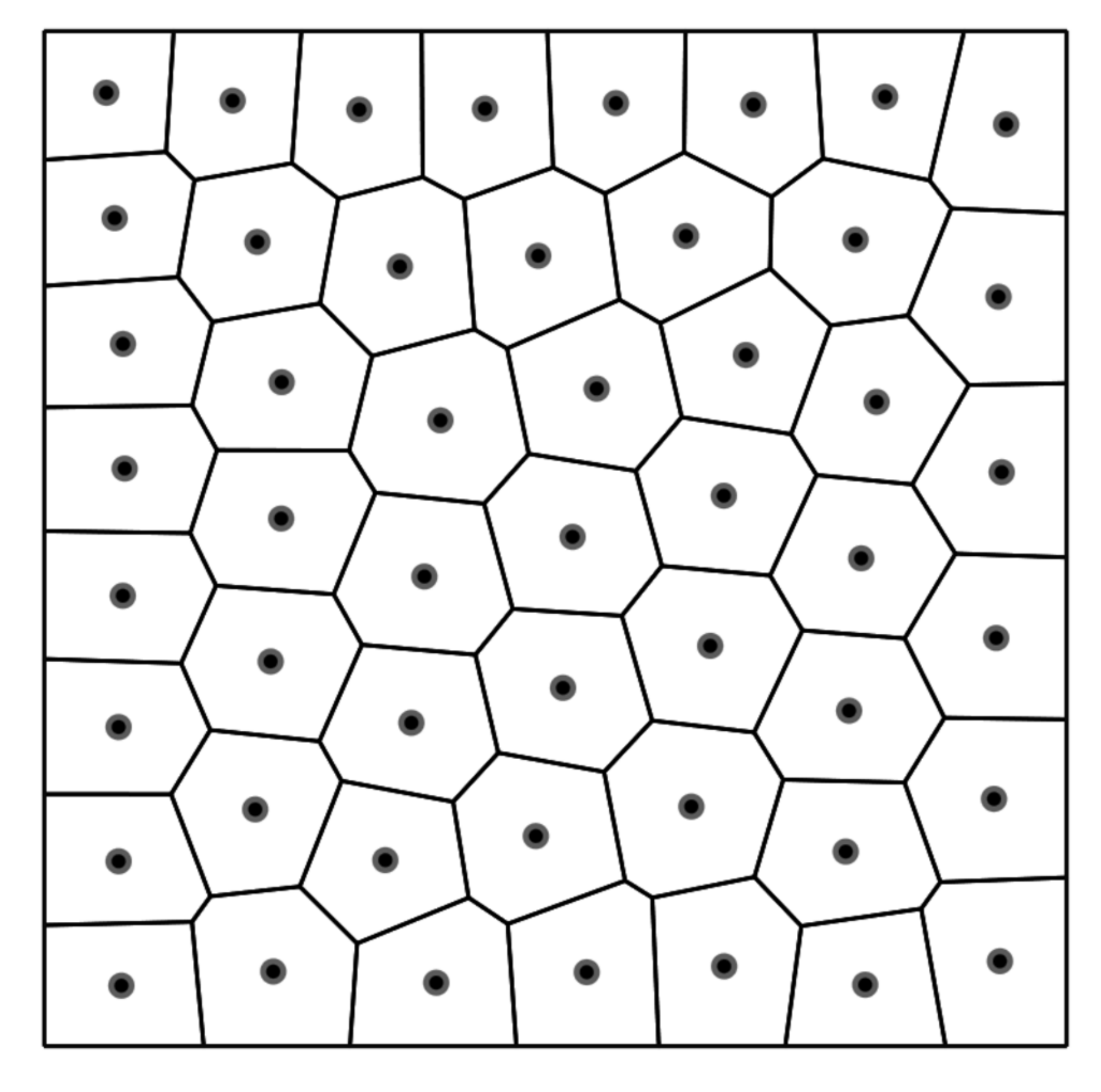}
      }
      \caption{
        Left: random points (black dots), Voronoi diagram and cell centroids (gray dots); 
        Right: a barycentric Voronoi diagram is a local minimizer of $Q$.
      }
      \label{fig:CVT}
    \end{minipage}
    \begin{minipage}{0.42\linewidth}
      \centerline{
        \includegraphics[width=40mm]{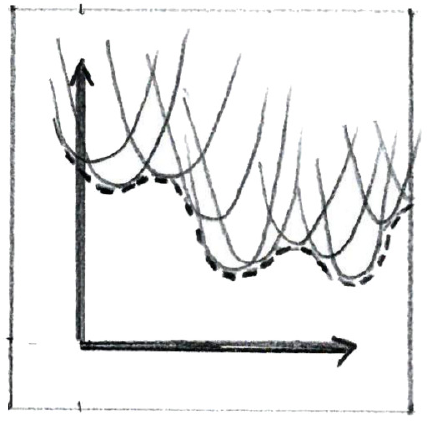}
      }
      \caption{The quantization noise power $Q$ minimized in vector
      quantization is a lower envelope. % of quadratic functions $Q_T$.
		 }
      \label{fig:LowerEnveloppeCVT}
    \end{minipage}
\end{figure}

\subsection{Relation with optimal sampling}
\label{sect:CVTandOT}

In this section, I exhibit some relations between
semi-discrete optimal transport and another problem referred to as
\emph{optimal sampling} (or \emph{vector quantization}).  Given a
compact $\Omega \subset \mathR^d$, a measure $\mu$, and a set of $k$
points $Y$ in $\mathR^d$, the \emph{quantization noise power} of $Y$
is defined as~:
\begin{equation}
Q(Y) := \int_\Omega \min_i \| x - y_i \|^2 d\mu = \sum\limits_{i=1}^k \int_{\Vor(y_i)} \| x - y_i \|^2 d\mu
\label{eqn:CVT}
\end{equation}

The quantization noise power measures how good $Y$ is at ``sampling''
$\Omega$ (the smaller, the better), see the survey in \cite{Du:1999:CVT:340312.340319}.  
The \emph{vector quantization problem} consists in minimizing $Q(Y)$
(i.e. finding the poinset $Y$ that best samples $\Omega$). This notion
comes from signal processing theory, and was used to find the optimal
assignment of frequency bands for multiplexing communications in a
single channel \cite{Lloyd82leastsquares}.  Designing a numerical
algorithm that optimizes $Q$ requires to evaluate the gradient of
$Q$. This requires computing integrals over varying domains (since the
Voronoi cells of the $y_i$'s depend on the $y_i$'s), which requires
several pages of careful derivations, as done in
\cite{Iri1984,Du:1999:CVT:340312.340319}. At the end, most of the
terms cancel-out, leaving a simple formula (see below). One can note
the similarity between the quantization noise power (Equation
\ref{eqn:CVT}) and the objective function maximized by the weight
vector in semi-discrete optimal transport (proof of Theorem
\ref{thm:SemiDiscrete}). This suggests using the same type of argument
(envelope theorem) to directly obtain the gradient of $Q$~:

\begin{observation}
The function $Q$ is of class $C^1$ (at least \footnote{it is in fact of class $C^2$ almost everywhere \cite{liu:onCVT:09}}) 
 and the components of its gradient relative to one of the point $y_i$ is given by:
$$
    \nabla_{y_i} Q(Y) = 2 m_i (y_i - g_i)
$$
where $m_i = \mu(\Vor(y_i)) = \int_{\Vor(y_i)}d\mu$ denotes the mass of the Voronoi cell $\Vor(y_i)$ and $g_i = 1/m_i \int_{\Vor(y_i)} x d\mu$ denotes
the centroid of the Voronoi cell $\Vor(y_i)$.
\end{observation}
\begin{proof}
  Consider the function $Q_T(Y) := \int_\Omega \| x - T(x) \|^2 d\mu$, parameterized by an 
  assignment $T: \Omega \rightarrow Y$. We are in a setting similar to semi-discrete
  optimal transport (Section \ref{sect:semidiscreteOT}), except that the function $Q_T(Y)$
  is quadratic (see Figure \ref{fig:LowerEnveloppeCVT}), whereas $F_T(W)$ is linear (Figure \ref{fig:LowerEnveloppe}).
   We have~:
\begin{itemize}
  \item $Q(Y) = Q_{T_{\Vor}}(Y)$;
  \item for a given $Y$, $T_{\Vor}$ is the unique affectation that minimizes $Q_T(Y)$.
\end{itemize}
By the envelope theorem, we have:
$$
\begin{array}{lcl}
   \nabla Q(Y) & = &  \nabla Q_{T_{\Vor}}(Y)) = \nabla \sum_i \int\limits_{\Vor(y_i)} (x^2 - 2x \cdot y_i + y_i^2) d\mu   \\[4mm]
               & = & \sum_i \left( \nabla \int\limits_{\Vor(y_i)}x^2 d\mu
                    - 2 \nabla \int\limits_{\Vor(y_i)} x \cdot y_i d\mu + \nabla \int\limits_{\Vor(y_i)} y_i^2 d\mu \right) \\[6mm]
   \nabla_{y_i}Q(y) & = & -2 y_i \int\limits_{\Vor(y_i)} x d\mu + 2 y_i \int\limits_{\Vor(y_i)} d\mu \\[4mm]
                   & = & -2 m_i g_i + 2 m_i y_i = 2 m_i (y_i - g_i)
\end{array}
$$
\end{proof}
This directly gives the expression of the gradient of $Q$ and explains why most of the terms cancel out in the derivations conducted
in \cite{Iri1984}. I mention that the same result can be obtained in a more general setting with Reynold's transport theorem \cite{nivoliers:AFM:2013}
(that deals with functions integrated over varying domains).

However, the envelope argument cannot be used to compute the Hessian of $Q$ (second order derivatives), and the 
structure of the formulas \cite{Iri1984,Du1999,liu:onCVT:09} do not suggest that direct computation can be avoided for them. Note
also that $Q$ is the lower envelope of a family of parabola (instead of a family of hyperplanes), therefore the concavity argument
does not hold, and the graph of $Q$ has many local minima (as depicted in Figure \ref{fig:LowerEnveloppeCVT}). The local minima of $Q$,
i.e. the point sets $Y$ such that $\nabla Q = 0$, satisfy $\forall i, y_i = g_i$, in other words, the position at each point $y_i$ corresponds
to the centroids of the Voronoi cell associated with $y_i$. For this reason, a stationary point of $Q$ is called a 
\emph{centroidal Voronoi tessellations}. To compute a centroidal Voronoi tessellation, it is possible to iteratively move each point
towards the centroid of its Voronoi cell (Lloyd relaxation \cite{Lloyd82leastsquares}), which is equivalent to minimizing $Q$ with
a gradient descent method \cite{Du:1999:CVT:340312.340319}. It is also possible to minimize $Q$ with Newton-type methods \cite{liu:onCVT:09} 
that show faster convergence. \\

More relations between semi-discrete optimal transport and vector quantization can be exhibited by considering a power diagram as the 
intersection between a $d+1$ Voronoi diagram and $\mathR^d$~:

\begin{observation}
  The $d$-dimensional power diagram $\Pow_W(Y)$ corresponds to the intersection between the $d+1$ dimensional Voronoi diagram 
$\Vor(\hat{Y})$ and $\mathR^d$, where the $\mathR^{d+1}$ lifting $\hat{y_i}$ of $y_i$ is defined by~:
\begin{small}
$$
  \hat{y_i} = \left(
\begin{array}{c}
  y_{i,1} \\
  y_{i,2} \\
\vdots \\
  y_{i,d} \\[2mm]
  h_i = \sqrt{w_M - w_i}
\end{array}  
  \right)
$$
\end{small}
where $y_{i,j}$ denotes the $j$-th coordinate of point $y_i$, and where $w_M$ denotes the maximum of all weights $\mbox{Max}(w_i)$.
\end{observation}
\begin{proof}
$$
\begin{array}{l}
   \Vor(\hat{y}_i) \cap \mathR^d  =  \{ x \quad | \quad \| \hat{x} - \hat{y}_i \|^2 < \| \hat{x} - \hat{y}_j \|^2 \ \forall j \neq i \} \\[2mm]
      \begin{array}{lcl}
                                & = &
             {\small   \left\{ x \quad | \quad 
                \left\| \left[\begin{array}{c} x \\[1mm] 0 \end{array}\right] - 
                        \left[\begin{array}{c} y_i \\[1mm] \sqrt{w_M - w_i} \end{array}\right] 
                \right\|^2  < 

                \left\| \left[\begin{array}{c} x \\[1mm] 0 \end{array}\right] - 
                        \left[\begin{array}{c} y_j \\[1mm] \sqrt{w_M - w_j} \end{array}\right] 
                \right\|^2 \ \forall j \neq i
               \right\} }  \\[6mm]
    & = & \{ x \quad | \quad \| x - y_i \|^2 - w_i + w_M < \| x - y_j \|^2 - w_j + w_M \ \forall j \neq i \} \\[2mm]
    & = & \{ x \quad | \quad \| x - y_i \|^2 - w_i < \| x - y_j \|^2 - w_j \ \forall j \neq i \} \\[2mm]
    & = & \Pow_W(y_i)
    \end{array}
\end{array}
$$
\end{proof}

We can now see a relation between vector quantization and semi-discrete optimal transport~:
\begin{observation}
   The quantization noise power $\hat{Q}(\hat{Y})$ computed in $\mathR^{d+1}$ corresponds to the 
   term $f_{T_W}(W)$ of the function maximized by the weight vector that defines a semi-discrete optimal 
   transport map plus the constant $w_M \mu(\Omega)$.
\end{observation}
\begin{proof}
$$
\begin{array}{lcl}
  \hat{Q}(\hat{Y}) & = & \sum\limits_i \int\limits_{\Vor(\hat{y}_i) \cap {\mathR^d}} \| \hat{x} - \hat{y}_i \|^2 d\mu \\[8mm]
                   & = & \sum\limits_i \int\limits_{\Pow_W(y_i)} \| x - y_i \|^2 - w_i + w_Md\mu  \\[4mm]
                   & = & f_{T_W}(W) + w_M \mu(\Omega)
\end{array}
$$
\end{proof}
The quantization noise power $Q$ is already known to be of class $C^2$ almost everywhere\footnote{
 ``by almost everywhere'', we mean that the function is no longer $C^2$ whenever two points become co-located, or
 whenever a Voronoi bisector matches a discontinuity of $\mu$ located on a straight line.} \cite{liu:onCVT:09}.
 As a consequence of this observation, since the function $f_{T_W}(W)$ can be obtained through the change of variable $h_i = \sqrt{w_M-w_i}$,
it is also of class $C^2$ almost everywhere. This gives more justification for using a quasi-Newton method to find the maximum
of $g$ as done in \cite{DBLP:journals/cgf/Merigot11} and in this paper (but note that a complete justification would require to find 
some bounds on the eigenvalue of the Hessian).

Another consequence of this observation is that given $\Omega \subset \mathR^d$, a measure $\mu$ and a pointset $Y$, optimizing $\hat{Q}$ for the first
$d$ coordinates moves the points in a way that minimizes the quantization noise power, and optimizing for the $d+1$ coordinate
computes the weights of a power diagram that defines an assignment that transports $\mu$ to the points. Interestingly, the       
first problem has multiple local minima, whereas the second one admits a global maximum.

%=================================================================================================

\section{Numerical Algorithm}
\label{sect:numerics}

I shall now explain how to use the results in Section \ref{sect:semidiscreteOT} and turn them into 
an efficient numerical algorithm. The algorithm is a variation of the one in \cite{DBLP:journals/cgf/Merigot11}.
Besides generalizing it to the 3d case, I make some observations that improve the efficiency of the multilevel 
optimization method. \\

The input of the algorithm is a measure $\mu$, represented by a simplicial complex $M$ (i.e. an interconnected set of tetrahedra in 3D),
a set $Y$ of $k$ points $y_i$ and $k$ masses $\nu_i$ such that $\sum \nu_i = \mu(M)$ where $\mu(.)$ is defined as follows~:
For a set $B \subset \mathR^3$, the measure $\mu(B)$ corresponds to the volume of the intersection between the tetrahedra of $M$ and
$B$. Optionally, $M$ can have a density linearly interpolated from its vertices. In this setting, the measure of $B$ corresponds to the integral
of the linearly interpolated density on the intersection between $B$ and the tetrahedra of $M$.

The weight vector that realizes the optimal transport can be obtained by maximizing the function $g(W)$ using different numerical
methods. The single-level version of the algorithm in \cite{DBLP:journals/cgf/Merigot11} is outlined in Algorithm \ref{alg:SingleLevel}~:

  \begin{algorithm}
    \SetAlgoLined
  \KwData{A tetrahedral mesh $M$, a set of points $Y$ and masses $\nu_i$ such that $\mu(M) = \sum\nu_i$}
  \KwResult{The weight vector $W$ that determines the optimal transport map $T$ from $M$ to $\sum \nu_i \delta_{y_i}$}
    \BlankLine
    $W \leftarrow 0$\\
    (1) \While{ $\| \nabla g(W) \|^2 < \epsilon\ $}{ 
    (2)  Compute $\Vor_W(Y) \cap M$ \\[2mm]
    (3)  Compute $g(W) = \sum_i \int\limits_{\Pow_W(y_i) \cap M} \| x - y_i \|^2 - w_i d\mu + \sum_i \nu_i w_i$ \\[2mm]
    (4)  Compute $\nabla g(W) = -\mu(\Pow_W(y_i)) + \nu_i$ \\
    (5)  update $W$ with L-BFGS 
    } 
    \BlankLine
    \caption{Semi-discrete optimal transport (single-level algorithm)}
    \label{alg:SingleLevel}
  \end{algorithm}

To facilitate reproducing the results, I give more details about each
step of the algorithm: (1): note that the components of the gradient
of $g$ correspond to the difference between the prescribed measures
$\nu$ and the measures of the power cells. This gives an
interpretation of the norm of the gradient of $g$, and helps
choosing a reasonable $\epsilon$ threshold. In the experiments below, I
used $\epsilon =  0.01 * \mu(M) / \sqrt{k}$. (2): the algorithm that computes the
intersection between a power diagram and a tetrahedral mesh is detailed further (Algorithm \ref{alg:VoroClip}). 
(3),(4): once the intersection $\Vor_W(Y) \cap M$
is computed, the terms $g(W)$ and $\nabla g(W)$ are obtained by summing the 
contributions of each intersection (grayed area in Figure \ref{fig:VoroClip}).
(5): To maximize $g$, as in \cite{DBLP:journals/cgf/Merigot11}, I use the L-BFGS numerical 
optimization method \cite{Liu:1989:LMB:81100.83726}. An implementation is available in \cite{LBFGSImpl}.

\subsection{Computing the intersection between a tetrahedral mesh and a power diagram}
\label{sect:RVD}

To adapt the 2d algorithm in \cite{DBLP:journals/cgf/Merigot11} to the 3d case, the only required component
is a method that computes the intersection between a tetrahedral mesh and a power diagram (step (2) 
in Algorithm \ref{alg:SingleLevel})~:

  \begin{algorithm}
    \SetAlgoLined
  \KwData{A tetrahedral mesh $M$, a set of points $Y$ and a weight vector $W$}
  \KwResult{The intersection $\Vor_W(Y) \cap M$}
    \BlankLine
    S: Stack(couple(tet index, point index)) \\
    \ForEach{ tetrahedron $t \in M$}{
      \If{ $t$ \emph{is not marked} }{
          (1) $i \leftarrow i \ |\ \Pow_W(y_i) \cap t \neq \emptyset$ \\
         Mark(t,i) \\
          Push(S, (t,i)) \\
          \While{\emph{S is not empty}}{
              (t,i) $\leftarrow$ Pop(S) \\
              (2) P: Convex $\leftarrow \Pow_W(y_i) \cap t$ \\
              (3) Accumulate(P) \\
              (4) \ForEach{j \emph{neighbor of} i \emph{in P}}{
                  \If{\emph{$(t,j)$ is not marked}}{
                      Mark($t,j$); \quad
                      Push(S, ($t,j$)) 
                  }
              }
              (5) \ForEach{$t^\prime$ \emph{neighbor of} $t$ \emph{in P}} {
                  \If{\emph{$(t^\prime,i)$ is not marked}}{
                      Mark($t^\prime,i$); \quad
                      Push(S, ($t^\prime,i$)) 
                  }
              }
          }
       }
    }
    \BlankLine
    \caption{Computing $\Pow_W(Y) \cap M$ by propagation}
    \label{alg:VoroClip}
  \end{algorithm}

The algorithm works by propagating simultaneously over the tetrahedra and the power cells. It traverses all
the couples $(t,i)$  such that the tetrahedron $t$ has a non-empty intersection with the power cell of $y_i$.
(1): Propagation is initialized by starting from
an arbitrary tetrahedron $t$ and a point $y_i$ that has a non-empty intersection between its power cell and $t$. 
I use the point $y_i$ that minimizes its power distance $\| y_i - . \|^2 - w_i$ to one of the vertices of $t$.
(2): a tetrahedron $t$ and a power cell $\Pow_W{y_i}$ can be both described as the intersection of half-spaces,
as well as the intersection $t \cap \Pow_W{y_i}$, computed using re-entrant clipping (each half-space is removed
iteratively). I use two version of the algorithm, a non-robust one that uses floating point arithmetics, and
a robust one \cite{PCK}, that uses arithmetic filters \cite{meyer:inria-00344297}, 
expansion arithmetics \cite{DBLP:conf/compgeom/Shewchuk96} and symbolic perturbation \cite{Edelsbrunner90simulationof}.
Both predicates and power diagram construction algorithm are available in PCK (Predicate
  Construction Kit) part of my publically available ``geogram''
programming library\footnote{\url{http://gforge.inria.fr/projects/geogram/}}. (3) the contribution of each
intersection $P = t \cap \Pow_W{y_i}$ is added to $g$ and $\nabla g$. 

The convex $P$ is illustrated in the (2d) figure \ref{fig:VoroClip}
as the grayed area (in 3d, $P$ is a convex polyhedron). The algorithm then propagates to both neighboring tetrahedra
and points. (4): each portion of a facet of $t$ that remains in $P$ triggers a 
propagation to a neighboring tetrahedron $t^\prime$. In the 2d example of Figure \ref{fig:VoroClip}, this corresponds to
edges $e_1$ and $e_4$ that trigger a propagation to triangles $t_2$ and $t_1$ respectively. (5): each facet of $P$ 
generated by a power cell facet triggers propagation to a neighboring point. In the 2d example of the figure, this
corresponds to edges $e_2$ and $e_3$ that trigger propagation to points $y_{j_1}$ and $y_{j_2}$ respectively. 

This algorithm is parallelized, by partitioning the mesh $M$ into $M_1$, $M_2$, \ldots $M_{nb\_cores}$ and by computing
in each thread $M_{thrd} \cap \Pow_W(Y)$. 

\begin{figure}
\includegraphics[width=100mm]{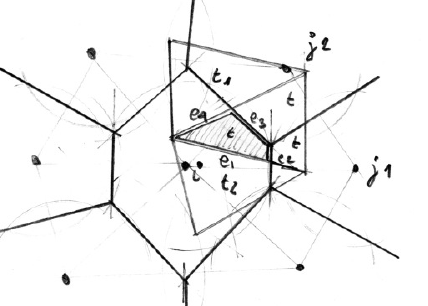}
\caption{Computing the intersection between a power diagram and a tetrahedral mesh by propagation.} 
\label{fig:VoroClip}
\end{figure}

\begin{table}
\begin{tabular}{l|lllllll}
   \hline
   \hline
   nb masses $k$    & 1000 & 2000 & 5000  & 10000 & 30000 & 50000 & 100000 \\
   nb iter          & 146  & 200  & 328   &  529  &  1240 & 1103  & 1102   \\
   time (s)         & 2.8  & 6.4  & 21    &  65   &  232  & 568   & 847   \\
   \hline 
\end{tabular}
\caption{Statistics for a simple translation scenario with the single-level algorithm.
 The threshold for $\| \nabla g \|^2$ is set to $\epsilon =  0.01 * \mu(M) / \sqrt{k}$.
}
\label{tab:SingleLevel}
\end{table}

I conducted a simple experiment, where $M$ is a tessellated sphere with 2026 tetrahedra, and $Y$ a sampling of the same sphere
shifted by a translation vector of three times the radius of the sphere. The statistics in Table \ref{tab:SingleLevel} obtained
with a standard PC\footnote{
experiments done with a 2.8 GHz Intel Core i7-4900MQ CPU with an implementation of Algorithm \ref{alg:VoroClip} that uses 8 threads.}
show that the single-level algorithm does not scale-up well with the number of points and starts taking a signiffficant time
for processing 10K masses and above. This confirms the observation in \cite{DBLP:journals/cgf/Merigot11}.
This is because at the initial iteration, all the weights are zero, and the
power diagram corresponds to the Voronoi diagram of the points $y_i$. At this step, only some points $y_i$ on the border
of the pointset have a Voronoi cell that ``see'' the mesh $M$ (i.e. that have a non-empty intersection with it). It takes
many iteration to compute the weights that ``shift'' the concerned power cells onto $M$ and allow inner points to see $M$.
It is only once all the points of $Y$ ``see'' $M$ that the numerical method can capture the trend of $g$ around the maximum
(and then it takes a small number of iterations to the algorithm to balance the weights).
Intuitively, $Y$ is ``peeled'' only one layer of points at a time. The bad effect on performances is even more important
than in \cite{DBLP:journals/cgf/Merigot11}, because in the 3d setting, the proportion of ``inner'' points relative
to the number of points on the border of the pointset is larger than in 2d. 

\subsection{Multi-level algorithm}

To improve performances, I follow the approach in \cite{DBLP:journals/cgf/Merigot11}, that uses a multilevel algorithm.
The idea consists in ``bootstrapping'' the algorithm on a coarse sub-sampling of the pointset. The ``peeling'' effect
mentioned in the previous paragraph is limited since we have a small number of points. Then the algorithm is run with
a larger number of points, using the previously computed weights as an initialization. The set of points can be
decomposed into multiple level of increasing resolution.  The complete algorithm will be detailed below (Algorithm \ref{alg:MultiLevel}).

\begin{table}
\begin{tabular}{l|lllllll}
   \hline
   \hline
   nb masses                & 1000 & 2000 & 5000  & 10000 & 30000  & 50000 & 100000 \\
    deg. 0 time (s)         & 2.5  & 6    & 19    & 38    & 184    & 356   & 959    \\
    deg. 1 time (s)         & 1    & 2    &  6    & 14    & 54     & 103    & 172    \\
    deg. 2 time (s)         & 1.4  & 2.2  &  6    & 16    & 58     & 138    & 172    \\
    BRIO/deg. 2 time (s)    & 1    & 1.65 &  3.4  & 9     & 26     & 62    & 106    \\
   \hline
   single level time (s)    & 2.8  & 6.4  & 21    &  65   &  232  & 568   & 847   \\
   \hline 
\end{tabular}
\caption{Statistics for a simple translation scenario with the multi-level algorithm. The mesh $M$ has 61233 tetrahedra. 
   Timings are in seconds. Each level is initialized from the previous one with regressions of different degrees.}
\label{tab:MultiLevel}
\end{table}

To further improve the speed of convergence, I use the remark in Section \ref{sect:SemiDiscKanto} that the weights
$w_i$ corresponds to the potential $\phi$ evaluated at $y_i$ (with a 1/2 factor). For a translation, 
we know that $T^{-1}(y) = y - V = y - \nabla \phi$, therefore $\phi(y) = V \cdot y$ where $V$ denotes the translation vector.
In more general settings, $\phi$ is still likely to be quite regular (except on its singularities where $T$
is discontinuous). When initializing a level from the previous one, this suggests initializing the new $w_i$'s 
from a regression of their nearest neighbors computed at the previous level. Table \ref{tab:MultiLevel}
shows the statistics for initialization with the nearest neighbor (deg. 0), linear regression with 10 nearest
neighbors (deg. 1) and quadratic regression with 20 nearest neighbors (deg. 2). As can be seen, initializing
with linear regression results in a significant speedup. In this specific case though, quadratic regression
does not gain anything. It is not a big surprise since we know already that $\phi(y) = V \cdot y$ is linear 
in this specific case, but it can slightly improve performances in more general settings, as shown further. 
Finally, it is possible to gain another x2 speedup factor~: the algorithm that
we use to compute the power diagrams \cite{Amenta:2003:ICC:777792.777824} sorts the points with a multilevel 
spatial reordering method, that makes it very efficient. It is possible to use the same multilevel spatial ordering
for both the numerical optimization and for computing the power diagrams  (BRIO/deg. 2 row in the table).
Since only the weights change during the iterations, this order needs to be computed once only, at the beginning 
of the algorithm. Note the overall 8x acceleration factor as compared to the single-level algorithm in Table \ref{tab:SingleLevel}
(repeated in the last row of Table \ref{tab:MultiLevel} to ease comparison).
The complete multi-level algorithm is summarized below~:

\begin{algorithm}
    \SetAlgoLined
  \KwData{A tetrahedral mesh $M$, a set of points $Y$ and masses $\nu_i$ such that $\mu(M) = \sum\nu_i$}
  \KwResult{The weight vector $W$ that determines the optimal transport map $T$ from $M$ to $\sum \nu_i \delta_{y_i}$}
    \BlankLine
    Apply a random permutation to the points $Y$ \\
    (1) Partition the interval of indices $[1,k]$ of $Y$ into $n_l$ intervals $[b_l, e_l]$ of increasing size \\
    \ForEach{\emph{level} $l$}{
        (2) Sort the points $y_{b_l} \ldots Y_{e_l}$ spatially \\
        (3) {\bf For each } $i$, $\nu_i \leftarrow |M|/e_l$ \\
        (4) Interpolate the weights $w_{b_l} \ldots w_{e_l}$ from the already computed weights $w_1 \dots w_{b_l - 1}$\\
        Optimize the weights using Algorithm \ref{alg:SingleLevel}
    }
    \BlankLine
    \caption{Semi-discrete optimal transport (multi-level algorithm)}
    \label{alg:MultiLevel}
\end{algorithm}

In my implementation, for step (1), the ratio between the number of points in a level and in the rest of the points
is set to 0.125. For the spatial sort in step (2), the algorithm, available in ``geogram'', was inspired by
the variant of the Hilbert sort implemented in \cite{SpatialSort}. (3): Before computing the optimal transport maps,
since the number of points changes at each level, the masses of the points need to be updated. 
At step (4), to determine the weight of 
a new point $w_i$, I use linear least squares with 10 nearest neighbors for degree 1 and 
quadratic least squares with 20 nearest neighbors for degree 2.

The influence of the degree of the regression is evaluated in Table \ref{tab:MultiLevel2} for
a configuration where a sphere is splitted into two spheres (first row in Figure \ref{fig:result1}). Unlike in the
previous translation case, in this configuration the potential $\phi$ is non-linear (see the deformations of the
spheres), and a higher degree regression slightly improves the speed of convergence for a large number of
points, since it captures more variations of $\phi$ and better initializes $W$.

\begin{table}
\begin{tabular}{l|lllllll}
   \hline
   \hline
   nb masses                & 1000 & 2000 & 5000  & 10000 & 30000 & 50000  & 100000 \\
    BRIO/deg. 1 time (s) & 1    & 1.7  &  3.5  & 9.8   & 25    &  61.7  & 122  \\
    BRIO/deg. 2 time (s) & 0.9  & 1.6  &  3.5  & 8.4   & 28.3  &  61.4  & 112 \\
   \hline 
\end{tabular}
\caption{Statistics for splitting a sphere into two spheres with the multi-level algorithm. Timings are in seconds.
   Each level is initialized from the previous one with regressions of different degrees.}
\label{tab:MultiLevel2}
\end{table}

\subsection{Using semi-discrete transport to approximate the transport between two tetrahedral meshes}

I now consider the case where the input is a pair of tetrahedral meshes $M$ and $M^\prime$. The goal
is now to generate a sequence of tetrahedral meshes that realize an approximation of the optimal
transport between $M$ and $M^\prime$. The algorithm is outlined below~:

%\newpage
\begin{algorithm}
    \SetAlgoLined
  \KwData{Two tetrahedral meshes $M$ and $M^\prime$, and $k$ the desired number of vertices in the result}
  \KwResult{A tetrahedral mesh $G$ with $k$ vertices and a pair of points $p_i^0$  and $p_i^1$ attached to each vertex. Transport
            is parameterized by time $t \in [0,1]$ with $p_i(t) = (1-t)p_i^0 + t p_i^1$.}
    \BlankLine
    (1) Sample $M^\prime$ with a set $Y$ of $k$ points  \\
    (2) Compute the weight vector $W$ that realizes the optimal transport between $M$ and $Y$ (Algorithm \ref{alg:MultiLevel})\\
    (3) Compute $E = \Del(Y)|M^\prime$ and $F = \Pow_W(Y)|M$ \quad ; \quad Tets(G) $\leftarrow E \cap F$ \\
    (4) \textbf{Foreach} $i \in [1\ldots k]$, $(p_i)^0 \leftarrow \mbox{centroid}(\Pow_W(y_i) \cap M) \quad ; \quad (p_i)^1 \leftarrow y_i$ \\
    \BlankLine
    \caption{Approximated optimal transport between two tetrahedral meshes}
    \label{alg:ApproxTransport}
\end{algorithm}

The different steps of this algorithm are implemented as follows: (1): to compute a homogeneous sampling, 
I initialize $Y$ with a centroidal Voronoi tessellation (see Section \ref{sect:CVTandOT}).
(3): the main difficulty consists in finding the discontinuities in $T$ and avoid generating tetrahedra that cross them. 
To detect the discontinuities in $T$, I consider that the Voronoi diagram $\Vor(Y)$ that samples $M^\prime$ evolves towards 
the power diagram $\Pow_W(Y)$ that samples $M$ (note that this evolution goes backwards, from $M^\prime$ to $M$). Thus, the
tetrahedra that are kept are those that are present both in the dual $\Del(Y)$ of $\Vor(Y)$ (Delaunay triangulation) 
and the dual $\Reg_W(Y)$ of $\Pow_W(Y)$ (regular weighted triangulation). (4) Finally, the geometry $p_i^0$ of each vertex 
of $G$ at initial time $t=0$ is determined as the centroid of the power cell $\Pow_W(y_i) \cap M$. The geometry $p_i^1$ 
at final time $t=1$ is simply $y_i$.

%=================================================================================================

\section{Results and conclusions}
\label{sect:results}

\begin{figure}
  \includegraphics[width=\textwidth]{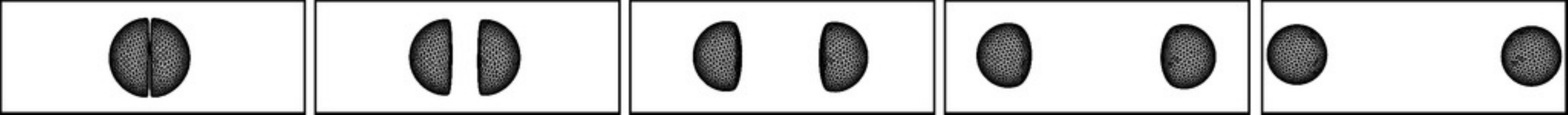}
  \includegraphics[width=\textwidth]{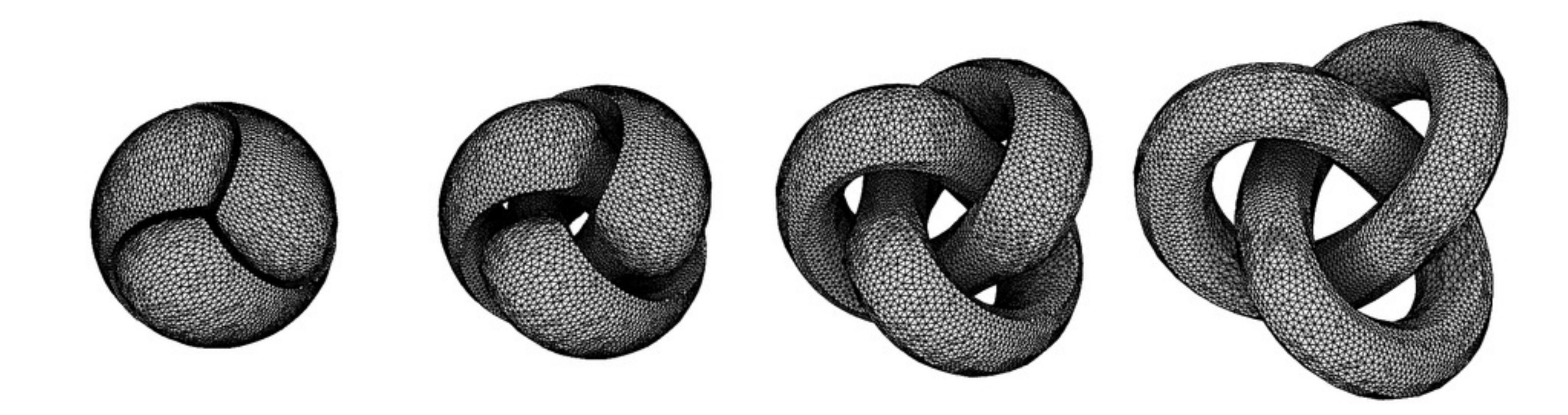}
  \includegraphics[width=\textwidth]{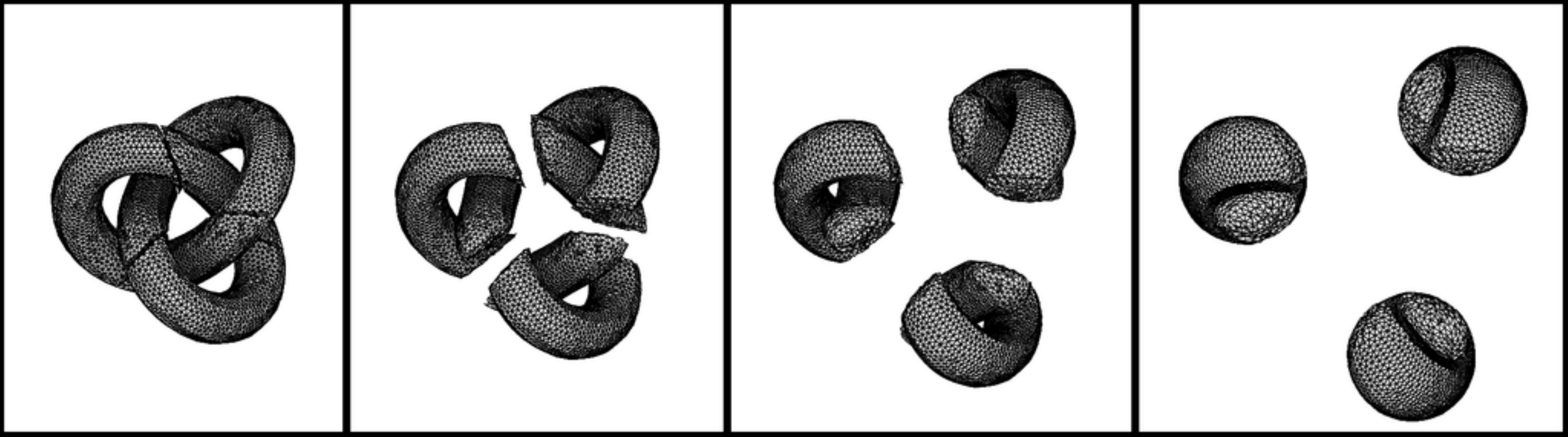}
\caption{Some examples of semi-discrete optimal transport with topology changes.}
\label{fig:result1}
\end{figure}

\begin{figure}
  \includegraphics[width=\textwidth]{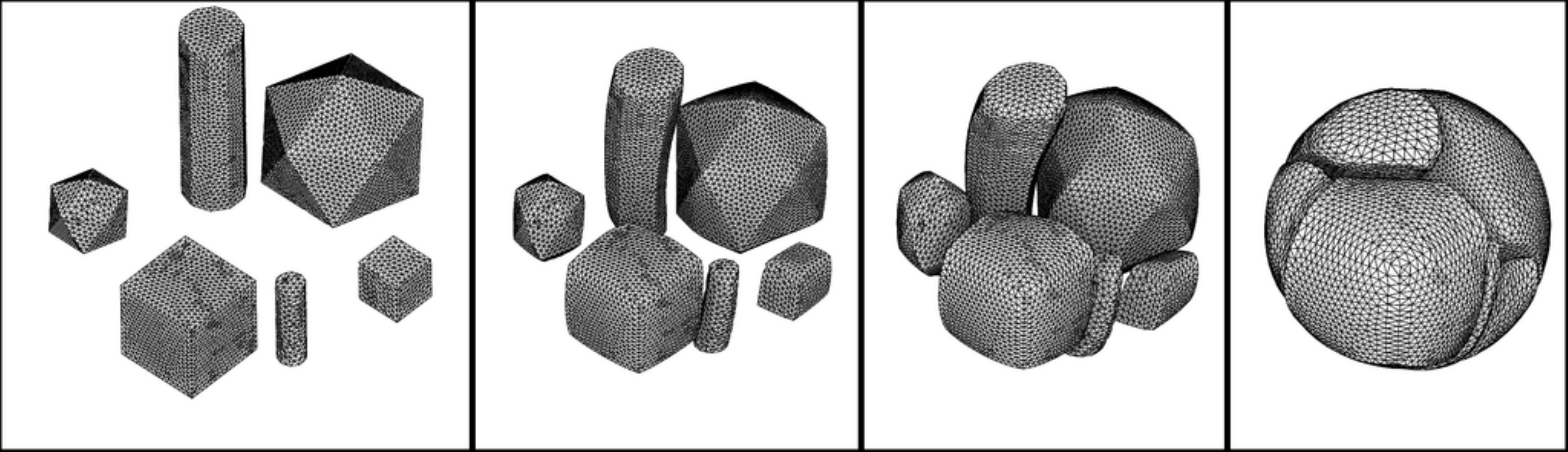}
  \includegraphics[width=\textwidth]{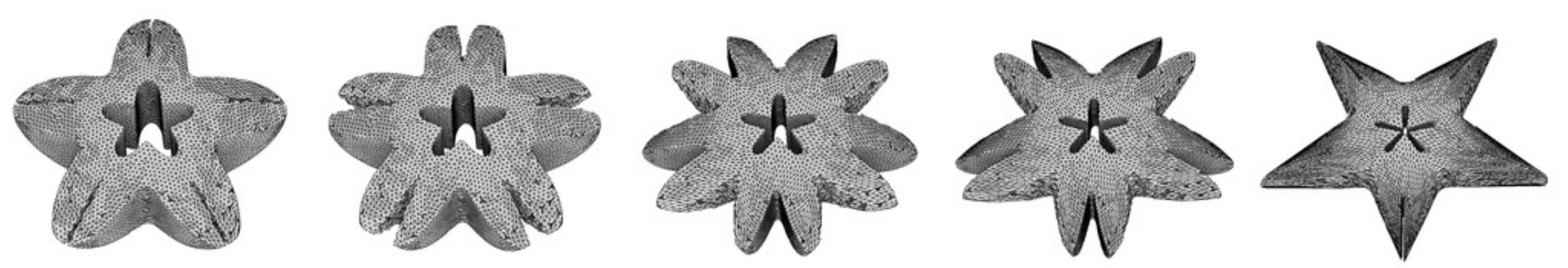}
  \includegraphics[width=\textwidth]{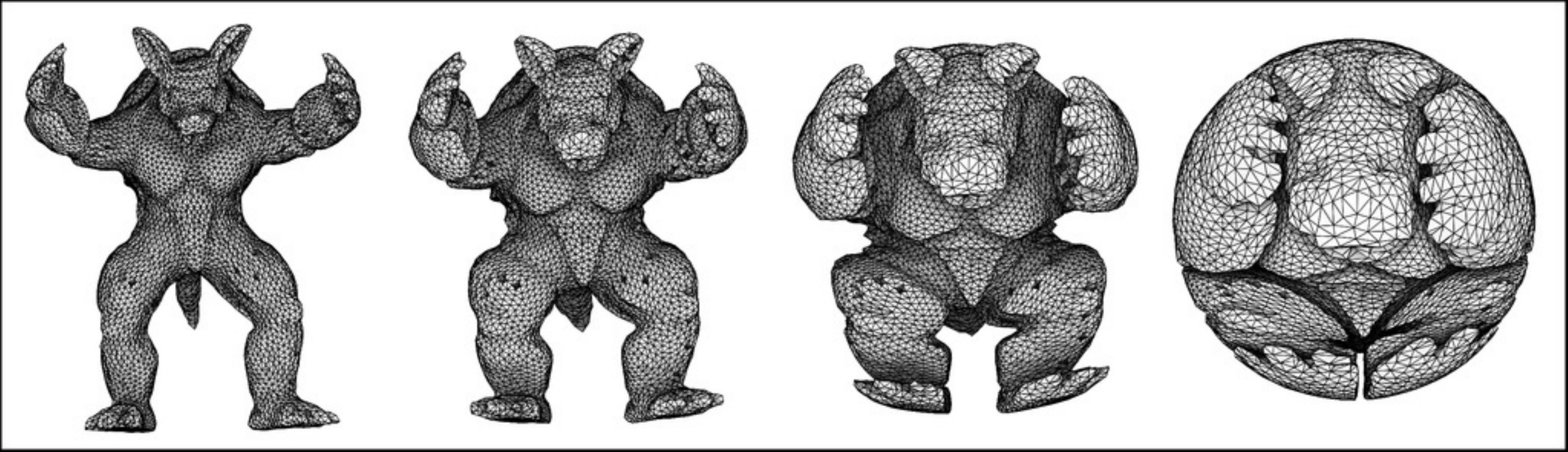}
\caption{More examples of semi-discrete optimal transport. Note how the solids deform
and merge to form the sphere on the first row, and how the branches of the star split
and merge on the second row. }
\label{fig:result2}
\end{figure}

\begin{table}
\begin{tabular}{l|llllllllll}
   \hline
   \hline
    nb masses   & 1000 & 2000 & 5000  & 10000 & 30000 & 50000  & $10^5$ & $3\times 10^5$ & $5\times 10^5$  & $10^6$ \\
    time (s)    & 1.45 & 3.2  & 7.3   & 17.3  & 55    & 154    &  187 &  671        & 1262          & 2649  \\
   \hline 
\end{tabular}
\caption{Statistics for the Armadillo $\rightarrow$ sphere optimal transport with varying number of masses (see third row 
   of Figure \ref{fig:result2}). Timings are given in seconds. The multi-level algorithm with BRIO pre-ordering and degree 2 regressions is used.}
\label{tab:Armadillo}
\end{table}

Several results are shown in Figures \ref{fig:result1} and
\ref{fig:result2}. Note that when the volume of $M$ and $M^\prime$ differ, using $\nu_i = |M|/k$
changes the ``density'' of $M^\prime$ and preserves the total mass. The intermediary steps 
are generated by using $p_i = (1-t) p_i^0 + t p_i^1$
for the locations at the vertices of $G$. As can be seen, the combinatorial criterion that 
selects the stable tetrahedra successfully finds the discontinuities. 
The third row of Figure \ref{fig:result2} demonstrates some potential applications in 
computer graphics. In the bottom row, the obtained deformation looks ``natural'' and ``visually pleasing''
(as far as I can judge, but my own judgment may be biased \ldots). However, a ``user'' would
probably prefer to rotate the star in the center column of Figure \ref{fig:result2}
rather than splitting and merging the branches, but optimal transport ``does not care'' about 
preserving topology.

Timings for the Armadillo $\rightarrow$ sphere optimal transport are given in Table \ref{tab:Armadillo}. The algorithm scales up reasonably
well, and computes the optimal transport from a tetrahedral mesh to 300K Dirac masses in 10 minutes. It scales-up to 1 million Dirac masses
(but it nearly takes 45 minutes).\\

To conclude, I mention that the main limitation of Algorithm \ref{alg:ApproxTransport} is that the discontinuities
are sampled at the precision of the initial sampling, that does not takes them into account. As a consequence, 
this leaves a gap that has a width of one
tetrahedron in the result. One can clearly see it in the figures. Moreover, when the shape undergoes strong deformations, flipping
may occur, making the concerned pairs of tetrahedra disappear in the result (for instance, one can observe 
some holes in the legs of the armadillo in Figure \ref{fig:result2}). With a better representation of discontinuity,
one may obtain a more precise representation of the transport. This leads to the following open questions,
that concern the continuous setting for some particular representations of $\mu$ and $\nu$~:

\begin{enumerate}
   \item  Given two tetrahedral meshes $M$ and $M^\prime$, is it possible to characterize the locus of the
  points where $T$ is discontinuous (\emph{discontinuity locus}), and invent an algorithm that generates 
  a faithful representation of it ?
   \item What does the discontinuity locus looks like if $M$ and $M^\prime$ both have a density linearly 
  interpolated over the tetrahedra ?
   \item What does the discontinuity locus looks like if $\mu$ and $\nu$ are supported by two different
  set of spheres ?
\end{enumerate}

\section*{acknowledgement}

I wish to thank Nicolas Bonneel for many discussions and for proofreading an early version of this article.

%=================================================================================================

\bibliographystyle{siam}
\bibliography{transport.bib}

\begin{thebibliography}{10}

\bibitem{OTuserguide}
{\sc L.~Ambrosio and N.~Gigli}, {\em A user’s guide to optimal transport},
  Modelling and Optimisation of Flows on Networks, Lecture Notes in
  Mathematics,  (2013), pp.~1--155.

\bibitem{Amenta:2003:ICC:777792.777824}
{\sc N.~Amenta, S.~Choi, and G.~Rote}, {\em Incremental constructions con
  brio}, in Proceedings of the Nineteenth Annual Symposium on Computational
  Geometry, SCG '03, New York, NY, USA, 2003, ACM, pp.~211--219.

\bibitem{DBLP:journals/siamcomp/Aurenhammer87}
{\sc F.~Aurenhammer}, {\em Power diagrams: Properties, algorithms and
  applications}, SIAM J. Comput., 16 (1987), pp.~78--96.

\bibitem{DBLP:conf/compgeom/AurenhammerHA92}
{\sc F.~Aurenhammer, F.~Hoffmann, and B.~Aronov}, {\em Minkowski-type theorems
  and least-squares partitioning}, in Symposium on Computational Geometry,
  1992, pp.~350--357.

\bibitem{ACFM:BB:2000}
{\sc J.-D. Benamou and Y.~Brenier}, {\em A computational fluid mechanics
  solution to the monge-kantorovich mass transfer problem}, Numerische
  Mathematik, 84 (2000), pp.~375--393.

\bibitem{DBLP:journals/tog/BonneelPPH11}
{\sc N.~Bonneel, M.~van~de Panne, S.~Paris, and W.~Heidrich}, {\em Displacement
  interpolation using lagrangian mass transport}, ACM Trans. Graph., 30 (2011),
  p.~158.

\bibitem{AP:BDM:2009}
{\sc R.~Burkard, M.~Dell'Amico, and S.~Martello}, {\em Assignment Problems},
  SIAM, 2009.

\bibitem{MAEintro}
{\sc L.~Caffarelli}, {\em The monge-amp\`ere equation and optimal
  transportation, an elementary review}, Optimal transportation and
  applications (Martina Franca, 2001), Lecture Notes in Mathematics,  (2003),
  pp.~1--10.

\bibitem{SpatialSort}
{\sc C.~Delage and O.~Devillers}, {\em Spatial sorting}, in CGAL User and
  Reference Manual. CGAL Editorial Board, 2011.
\newblock 3.9 edition.

\bibitem{Du:1999:CVT:340312.340319}
{\sc Q.~Du, V.~Faber, and M.~Gunzburger}, {\em Centroidal voronoi
  tessellations: Applications and algorithms}, SIAM Rev., 41 (1999),
  pp.~637--676.

\bibitem{Du1999}
{\sc Q.~Du, V.~Faber, and M.~Gunzburger}, {\em Centroidal {V}oronoi
  tessellations: applications and algorithms}, SIAM Review, 41 (1999),
  pp.~637--676.

\bibitem{Edelsbrunner90simulationof}
{\sc H.~Edelsbrunner and E.~P. M\"ucke}, {\em Simulation of simplicity: A
  technique to cope with degenerate cases in geometric algorithms}, ACM TRANS.
  GRAPH, 9 (1990), pp.~66--104.

\bibitem{DiscreteMA}
{\sc X.~Gu, F.~Luo, J.~Sun, and S.-T. Yau}, {\em Variational principles for
  minkowski type problems, discrete optimal transport, and discrete
  monge-ampere equations}, arXiv,  (2013).
\newblock [math.PR] http://arxiv.org/abs/1302.5472.

\bibitem{Iri1984}
{\sc M.~Iri, K.~Murota, and T.~Ohya}, {\em A fast {V}oronoi-diagram algorithm
  with applications to geographical optimization problems}, in Proc. IFIP,
  1984, pp.~273--288.

\bibitem{PCK}
{\sc B.~L\'evy}, {\em Restricted voronoi diagrams for (re)-meshing surfaces and
  volumes}, in Curves and Surfaces conference proceedings, 2014.

\bibitem{Liu:1989:LMB:81100.83726}
{\sc D.~C. Liu and J.~Nocedal}, {\em On the limited memory bfgs method for
  large scale optimization}, Math. Program., 45 (1989), pp.~503--528.

\bibitem{LBFGSImpl}
{\sc Y.~Liu}, {\em {HLBFGS}, a hybrid l-bfgs optimization framework which
  unifies l-bfgs method, preconditioned l-bfgs method, preconditioned conjugate
  gradient method.}
\newblock
  http://research.microsoft.com/en-us/um/people/yangliu/software/HLBFGS/.

\bibitem{liu:onCVT:09}
{\sc Y.~Liu, W.~Wang, B.~L\'{e}vy, F.~Sun, D.-M. Yan, L.~Lu, and C.~Yang}, {\em
  On centroidal {V}oronoi tessellation---energy smoothness and fast
  computation}, ACM Transactions on Graphics, 28 (2009), pp.~1--17.

\bibitem{Lloyd82leastsquares}
{\sc S.~P. Lloyd}, {\em Least squares quantization in pcm}, IEEE Transactions
  on Information Theory, 28 (1982), pp.~129--137.

\bibitem{DBLP:journals/focm/Memoli11}
{\sc F.~M{\'e}moli}, {\em Gromov-wasserstein distances and the metric approach
  to object matching}, Foundations of Computational Mathematics, 11 (2011),
  pp.~417--487.

\bibitem{DBLP:journals/cgf/Merigot11}
{\sc Q.~M{\'e}rigot}, {\em A multiscale approach to optimal transport}, Comput.
  Graph. Forum, 30 (2011), pp.~1583--1592.

\bibitem{meyer:inria-00344297}
{\sc A.~Meyer and S.~Pion}, {\em {FPG: A code generator for fast and certified
  geometric predicates}}, in {Real Numbers and Computers}, Santiago de
  Compostela, Espagne, 2008, pp.~47--60.

\bibitem{RePEc:ecm:emetrp:v:70:y:2002:i:2:p:583-601}
{\sc P.~Milgrom and I.~Segal}, {\em {Envelope Theorems for Arbitrary Choice
  Sets}}, Econometrica, 70 (2002), pp.~583--601.

\bibitem{Monge1784}
{\sc G.~Monge}, {\em M\'emoire sur la th\'eorie des d\'eblais et des remblais},
  Histoire de l'Académie Royale des Sciences (1781),  (1784), pp.~666--704.

\bibitem{nivoliers:AFM:2013}
{\sc V.~Nivoliers and B.~L\'evy}, {\em Approximating functions on a mesh with
  restricted voronoi diagrams}, in ACM/EG Symposium on Geometry Processing /
  Computer Graphics Forum, 2013.

\bibitem{papadakis:hal-00816211}
{\sc N.~Papadakis, G.~Peyr{\'e}, and E.~Oudet}, {\em {Optimal Transport with
  Proximal Splitting}}, SIAM Journal on Imaging Sciences, 7 (2014),
  pp.~212--238.

\bibitem{OTintro}
{\sc F.~Santambrogio}, {\em {Introduction to Optimal Transport Theory}}, arXiv,
   (2010).
\newblock [math.PR] http://arxiv.org/abs/1009.3856.

\bibitem{DBLP:conf/compgeom/Shewchuk96}
{\sc J.~R. Shewchuk}, {\em Robust adaptive floating-point geometric
  predicates}, in Symposium on Computational Geometry, 1996, pp.~141--150.

\bibitem{OTON}
{\sc C.~Villani}, {\em Optimal transport : old and new}, Grundlehren der
  mathematischen Wissenschaften, Springer, Berlin, 2009.

\end{thebibliography}
\end{document}